\documentclass[10pt]{amsart}

\usepackage{amsfonts,amssymb}
\newtheorem{proposition}{Proposition}[section]
\newtheorem{lemma}[proposition]{Lemma}
\newtheorem{corollary}[proposition]{Corollary}
\newtheorem{theorem}[proposition]{Theorem}
\theoremstyle{definition}

\newtheorem{example}[proposition]{Example}

\theoremstyle{remark}
\newtheorem{remark}[proposition]{Remark}
\newtheorem{remarks}[proposition]{Remarks}
\newcommand{\thlabel}[1]{\label{th:#1}}

\newcommand{\thref}[1]{Theorem~\ref{th:#1}}
\newcommand{\selabel}[1]{\label{se:#1}}

\newcommand{\lelabel}[1]{\label{le:#1}}
\newcommand{\leref}[1]{Lemma~\ref{le:#1}}

\newcommand{\colabel}[1]{\label{co:#1}}
\newcommand{\coref}[1]{Corollary~\ref{co:#1}}

\newcommand{\relabel}[1]{\label{re:#1}}
\newcommand{\reref}[1]{Remark~\ref{re:#1}}

\newcommand{\reslabel}[1]{\label{res:#1}}
\newcommand{\resref}[1]{Remarks~\ref{res:#1}}

\newcommand{\eqlabel}[1]{\label{eq:#1}}
\newcommand{\equref}[1]{(\ref{eq:#1})}

\newcommand\smi{\mbox{$S^{-1}$}}
\def\ot{\otimes}
\def\va{\varepsilon}

\def\l{\lambda}

\def\va{\varepsilon}
\def\v{\varphi}

\def\rh{\rightharpoonup}
\def\lh{\leftharpoonup}

\def\a{\alpha}
\def\b{\beta}

\def\cal{\mathcal}

\def\Inn{\mbox{\rm Inn}}
\def\Id{\mbox{\rm Id}}

\newcommand{\hrat}{H^{*\rm rat}}
\def\equal#1{\smash{\mathop{=}\limits^{#1}}}

\begin{document}
\title[$S^4$ for Co-Frobenius (C)QT Hopf algebras]
{On the antipode of a co-Frobenius (co)quasitriangular Hopf algebra}
\author{Margaret Beattie}
\address{Department of Mathematics and Computer Science, Mount Allison
University,\newline Sackville, NB E4L 1E6, Canada}
\email{mbeattie@mta.ca}\thanks{The first author's research is
supported by NSERC}
\author{Daniel Bulacu}
\address{Faculty of Mathematics and Informatics, University
of Bucharest, Str. Academiei 14, RO-010014 Bucharest 1, Romania}
\email{dbulacu@al.math.unibuc.ro}
\thanks{The second author had support from a PDF at Mount Allison University.
He thanks Mount Allison University (Canada) for their warm
hospitality.}

\begin{abstract}
We extend to the co-Frobenius case a result of Drinfeld and Radford
related to the fourth power of the antipode of a finite dimensional
(co) quasitriangular Hopf algebra.
\end{abstract}
\maketitle

\section{Introduction and Preliminaries}\selabel{0}
Throughout this paper, $H$ will denote a (not necessarily finite
dimensional) co-Frobenius Hopf algebra over a field $k$. All maps
are assumed to be $k$-linear. We use the conventional
Sweedler-Heyneman notation for the Hopf algebra comultiplication:
$\Delta (h)=h_1\ot h_2$, $h\in H$ (summation understood). As
usual, the $H^*$-bimodule structure on $H$ and the $H$-bimodule
structure on $H^*$ are given by
$$
l^* \rightharpoonup h \leftharpoonup m^* = m^*(h_1)h_2 l^*(h_3)
\mbox{  and  }
(h \rightharpoonup m^* \leftharpoonup  l)(m) = m^*(lmh)
$$
for all $h,l,m \in H$,  $l^*,m^* \in H^*$. The antipode of $H$ is
denoted $S$ with composition inverse $S^{-1}$. The group of
grouplike elements of $H$ is denoted $G(H)$ and the  grouplikes of
$H^0$, namely the set of algebra maps from $H$ to $k$, by $G(H^0)$.
We assume familiarity  with the basic theory of Hopf algebras; see
\cite{dnr, mo, sw} for example.\vspace{1.5mm}
\par In this short note, we see that for $H$ co-Frobenius and
quasitriangular (coquasitriangular), the grouplike elements in $H$
($H^*$ respectively) which define $S^4$ as an inner (co-inner)
automorphism can be expressed in terms of the modular elements in
$H$ and $H^*$, thus extending results of Drinfeld and Radford.
\vspace{1.5mm}
\par Recall that a Hopf algebra $H$ is co-Frobenius if
$\hrat$, the unique maximal rational submodule of $H^*$, is nonzero,
or, equivalently, if  the space of left or right integrals for $H$,
denoted $\int_l^{H^*} $ and $\int_r^{H^*}$ respectively, is nonzero.
 It was shown in \cite{bdr} that $H$ contains a distinguished
grouplike element $a$, which we also call the modular element of
$H$, such that for all $\lambda \in \int_l^{H^*}$,
\begin{equation}\eqlabel{a}
\lambda(h_1)h_2 = \lambda(h)a^{-1}
\mbox{  and  }
\lambda \circ S^2 = a^{-1} \rightharpoonup \lambda \leftharpoonup a.
\end{equation}
From the definition of a left integral and from \equref{a}, we have
that for all $h,l \in H$,
\begin{equation}\eqlabel{b}
\lambda(hl_2)l_1 = \lambda(h_2l)S(h_1) \mbox{ and  }
\lambda(hl_1)l_2 = \lambda(h_1l)S^{-1}(h_2)a^{-1}.
\end{equation}
\par For $\Gamma $ either a nonzero left or right integral for $H$
in $\hrat$, there are bijective maps from $H$ to $\hrat$ given by
\begin{equation}\eqlabel{bijection}
h \mapsto (h \rightharpoonup \Gamma) \mbox{  and  } h\mapsto (\Gamma
\leftharpoonup h).
\end{equation}

Let $\chi$ denote the generalized Frobenius automorphism of $H$
defined in \cite{bbt}, that is, for $\lambda \in \int_l^{H^*}$,
$\chi$ is the algebra automorphism of $H$ defined by
\begin{equation}\eqlabel{fx1i}
h\rh \l =\l \lh \chi(h),\mbox{ for all } h \in H.
\end{equation}

Then the algebra map  $\alpha :=\varepsilon \circ \chi \in H^*$ is
called the modular element for $H$ in $H^*$, terminology  justified
by the finite dimensional case; see \cite[Remark 2.2]{bbt}.
Moreover,
\begin{equation}\eqlabel{fx1ii}
\chi(h)=\alpha(h_2)S^{-2}(h_1) \mbox{ for all } h \in H.
\end{equation}

  Radford's formula for $S^4$ extends
to co-Frobenius Hopf algebras as follows.
\begin{theorem}\cite[Theorem 2.8]{bbt}\label{S4}
For $H$ co-Frobenius with modular elements $a \in H$,
$\alpha \in H^*$, then for all $h \in H$,
\begin{equation*}
S^4(h) = a(\alpha \rightharpoonup h \leftharpoonup \alpha^{-1})a^{-1} =
\alpha^{-1}(h_1)ah_2a^{-1}\alpha(h_3).
\end{equation*}
\end{theorem}
Of course, if $H$ is not finite dimensional, then there is no
nonzero integral $t\in H$, but nonetheless many of the properties of
finite dimensional Hopf algebras   which are proved using integrals
(such as the $S^4$ formula above) carry over to infinite dimensional
co-Frobenius Hopf algebras using the integral in $H^*$ and the
modular elements. Another example, although not directly related to
the theme of this paper, is given in the next proposition as further
illustration of this point of view.
\begin{proposition}\label{tangent}(cf. \cite[Proposition
4]{radford}) Let $H$ be co-Frobenius with $\lambda, a, \alpha$ as
above. Then the following are equivalent. \par (i)  $S^2$ is
co-inner (or equivalently $S^{-2}$ is co-inner).
\par (ii) There are $\rho, \tau \in (H \otimes H)^*$ such that for
all $h,l \in H$, \begin{equation} \eqlabel{c} \lambda(lh) = \rho(h_1
\otimes l_1)\lambda(h_2  l_2)\tau(h_3 \otimes l_3).\end{equation}
\end{proposition}
\begin{proof}We outline the proof.
 If $S^{-2} = {\rm co-Inn}_\omega$,
i.e., $S^{-2}(h) = \omega^{-1}(h_1)h_2 \omega(h_3)$ for some $\omega
\in \cal{U}(H^*)$, then writing $\lambda(lh) = \lambda(\chi(h)l)$,
and using \equref{fx1ii}, it is easy to see that $\lambda(lh) =
\rho(h_1 \otimes l_1)\lambda(h_2l_2) \tau(h_3 \otimes l_3)$ where
$\rho = \omega^{-1} \otimes \varepsilon$ and $\tau = (\omega \otimes
\varepsilon)*(\alpha \otimes \varepsilon)$.
\par Conversely, note that $\lambda(lh) = (\lambda \leftharpoonup \chi(h))(l)$
and, use \equref{b} to obtain $$\rho(h_1 \otimes
l_1)\lambda(h_2l_2)\tau(h_3 \otimes l_3) = \rho(h_1 \otimes
S(h_2))\lambda(h_3l)\tau(h_5 \otimes S^{-1}(h_4)a^{-1}).$$  Thus
from  \equref{fx1i}  together with \equref{bijection}, one obtains
that
$$\chi(h) = \alpha(h_2)S^{-2}(h_1) = \rho(h_1 \otimes S(h_2))h_3
\tau(h_5 \otimes S^{-1}(h_4)a^{-1}) = \rho'(h_1)h_2 \tau'(h_3),$$
where $\rho',\tau'$ have the obvious definitions. Thus $S^{-2}(h) =
\rho'(h_1)h_2 \tau''(h_3)$ where $\tau''= \tau' * \alpha^{-1}$.
Applying $\varepsilon$, we see that $\rho' * \tau'' = \varepsilon$.
Now compute $\rho' \rightharpoonup S^{-2}(h) = h \leftharpoonup
\rho'$ to see that $\rho' = \rho'\circ S^{-2}$ and similarly $\tau''
= \tau'' \circ S^{-2}$. Then $\tau'' * \rho'(h) = \tau'' *
\rho'(S^{-2}(h))= \varepsilon(h)$ and so $\rho'$ and $\tau''$ are
inverse in $H^*$ and
  $S^{-2}$, and thus $S^2$, is co-inner.
\end{proof}

For $(H, R)$   quasitriangular then $S^2$ is an    inner
automorphism  induced by  $u$ or by $v:=S(u)^{-1}$ defined via
  $R $. Then $vu=uv$, and $S^4=\Inn_{vu}$. As well,
  (see  \equref{deltau} below),  $vu\in G(H)$. For $H$
finite dimensional,  it was proved independently by Drinfeld
\cite[Proposition 6.2]{drinfeld} and Radford \cite[Theorem
2]{radford} that $uv = vu = ag_\alpha$ where $a \in H$ and $\alpha
\in H^*$ are the modular elements as above, and $g_\alpha$ is a
grouplike element defined via
  $R$ and
$\alpha $. In this paper, we extend this formula  to not necessarily
finite dimensional quasitriangular co-Frobenius Hopf algebras with
an analogous result for coquasitriangular co-Frobenius Hopf
algebras. Note that this paper takes a new approach and that the
proof in the coquasitriangular case cannot be viewed as the formal
dual of the proof given in the quasitriangular case.

\section{$S^4$ when $(H, R)$ is a co-Frobenius quasitriangular Hopf algebra}\selabel{2}
\setcounter{equation}{0}  Throughout this section, as well as
co-Frobenius,  $(H, R)$ will be almost cocommutative or else
quasitriangular   with modular elements $a\in H$ and $\alpha \in
H^*$. For example, the tensor product of a finite dimensional
quasitriangular Hopf algebra and an infinite dimensional group
 algebra would satisfy the above conditions.

\par Recall that, for $H$ a Hopf algebra and
$R=R^1\otimes R^2=r^1 \otimes r^2\in \cal{U}(H \otimes H)$,
then $(H, R)$ is called almost cocommutative if for all $h\in H$,
\begin{equation} \eqlabel{qt4}
\Delta^{\rm cop}(h) = R \Delta(h) R^{-1}.
\end{equation}

Then by results of  Drinfeld \cite[Proposition 2.2]{drinfeld} (who
credits Lyubashenko) or Radford \cite[Proposition 1]{radford}, $S^2$
is an inner automorphism induced by $u:=S(R^2)R^1$ or by $v:=
S(u)^{-1}= S(U^1)U^2$ where $U:=R^{-1}= U^1 \otimes U^2$, i.e.,
$$
S^2(h) = uhu^{-1} = vhv^{-1} \mbox{ for all } h \in H.
$$

\begin{remark}\relabel{2.1}
If $S^2 = \Inn_w$, so that $\lambda \circ S^2 = a^{-1}
\rightharpoonup \lambda \leftharpoonup a = w^{-1} \rightharpoonup
\lambda \leftharpoonup w$, by \equref{fx1ii} we have
$\chi(w^{-1})w=\chi(a^{-1})a =\alpha(a^{-1})1$. If  $\va(w)=1$ then
$\a(w^{-1})=\a(a^{-1})$. \qed
\end{remark}

If, as well, the element $R$ satisfies the following,
\begin{eqnarray}
&&\Delta (R^1)\ot R^2=R^1\ot r^1\ot R^2r^2,\eqlabel{qt1}\\
&&R^1\ot \Delta (R^2)=R^1r^1\ot r^2\ot R^2,\eqlabel{qt2}\\
&&\va(R^1)R^2=1,~~\va(R^2)R^1=1, \eqlabel{qt3}
\end{eqnarray}
then $(H, R)$ is called quasitriangular. In this case,
\begin{equation}\eqlabel{invr}
U =S(R^1)\ot R^2=R^1\ot \smi (R^2), \mbox{ so that }  (S\ot S)(R)=R.
\end{equation}

Furthermore, for $u = S(R^2)R^1$ as defined above,
\begin{equation}\eqlabel{deltau}
\Delta(u)=(u\ot u)(R_{21}R)^{-1}=(R_{21}R)^{-1}(u\ot u), \mbox{
where  } R_{21}:= R^2\ot R^1,
\end{equation}
so that $uv = u S(u)^{-1} \in G(H)$.
\begin{remarks}\reslabel{inner}
(i) We note that if a co-inner automorphism $\v$ of an almost
cocommutative Hopf algebra $(H,R)$   is induced by a grouplike in
$H^0$ then $\varphi$ is inner.  Let $\gamma \in G(H^0)$  with
convolution inverse $\gamma^{-1}$ and let $\varphi \in {\rm Aut}(H)$
be the co-inner automorphism induced by $\gamma^{-1}$ as in
Proposition \ref{tangent}. Then
$$
\varphi(h) = \gamma^{-1}(h_1)h_2 \gamma(h_3) = w h w^{-1}
$$
where $w\in  W_\gamma = \{\gamma^{-1}(U^1)U^2, \gamma(R^1)R^2,
\gamma(U^2)U^1, \gamma^{-1}(R^2)R^1 \}$. To see this, note that by
\equref{qt4},
\[
\varphi(h) = [\gamma^{-1}(U^1h_2 R^1)U^2h_1R^2] \gamma(h_3) =
\gamma^{-1}(U^1)U^2 h \gamma^{-1}(R^1)R^2.
\]
The rest of the statement is shown by using \equref{qt4} in a
similar manner.

\par (ii) If $(H, R)$ is quasitriangular, the map from $G(H^0)$ to $G(H)$
given by
\[
\eta \mapsto a_{\eta}:= \eta(R^1)R^2
\]
is a group homomorphism. Also recall that $(H, \tilde{R})$ is
quasitriangular where $\tilde{R}:= R_{21}^{-1} = R^2\otimes S(R^1) =
S^{-1}(R^2) \otimes R^1$, and so there is also a group homomorphism
from $G(H^0)$ to $G(H)$ given by
\[
\eta \mapsto b_\eta := \eta(S^{-1}(R^2))R^1 = \eta^{-1}(R^2)R^1.
\]
(The proof of this is explicit in \cite[Proposition 3]{radford}.) So
in the quasitriangular case, using \equref{invr} we see that
$W_\gamma =\{a_\gamma, b_\gamma\}$. \par(iii) As in \cite{drinfeld},
the map from $G(H^0)$ to $G(H)$ defined by $\eta \mapsto a_\eta
b_{\eta^{-1}}$ is a group homomorphism with image in $G(H) \cap
Z(H)$.  We will see in the next lemma that under this group
homomorphism $\alpha$ maps to 1. For $H$ finite dimensional, this
fact was used by Radford \cite{radford2} in the proof that a
factorizable Hopf algebra is unimodular.\qed
\end{remarks}

For the remainder of this  section, we assume that $(H, R)$ is
quasitriangular. Then $S^4 = \Inn_{uv}   = \Inn_{aa_\alpha} =
\Inn_{a b_\alpha}$ and we show that the elements inducing the inner
automorphism $S^4$ are all equal.

\begin{lemma}\lelabel{factunim}
Let $H$ be a co-Frobenius quasitriangular Hopf algebra with modular
elements $a \in H$ and $\alpha \in H^*$. Then
$a_{\alpha}=b_{\alpha}$.
\end{lemma}
\begin{proof}
From \reref{2.1}, $\chi(u^{-1})u = \alpha(a^{-1})1$ and $\a(u^{-1})=\a(a^{-1})$.
Now, from \equref{fx1i} and \equref{deltau}, we have
\[
\chi(u^{-1})u=\alpha(u^{-1}_2)S^{-2}(u^{-1}_1)u=\alpha(R^1r^2u^{-1})
S^{-2}(R^2r^1u^{-1})u=
\alpha(a^{-1})a_{\alpha}b_{\alpha^{-1}}.
\]
Since $\alpha(a^{-1})\neq 0$ it follows that $a_{\alpha}=b_{\alpha}$,
as needed.
\end{proof}

Now we prove the main result of this section.


\begin{theorem}\thlabel{2.2}
Let $(H, R)$ be a co-Frobenius quasitriangular Hopf algebra with
modular elements $a\in H$ and $\alpha\in H^*$. Then for $u$ and
$v = S(u)^{-1}$ as above,
\begin{equation}\eqlabel{drcFqt}
uv = vu=ab_{\alpha}=aa_{\alpha }.
\end{equation}
\end{theorem}
\begin{proof}
By \leref{factunim} it suffices to show that $uv=ab_\a$.

Let $h\in H$ and $0 \neq \l \in \int_l^{H^*}$. Then, on one hand, we
have
\begin{eqnarray*}
\l(R^2h_2)R^1h_1
&\equal{\equref{qt2}}&
\l(h_1R^2)h_2R^1\\
&=&\l(h_1R^2_1)h_2R^2_2S(R^2_3)R^1\\
&\equal{\equref{a}}&
l(hR^2_1)a^{-1}S(R^2_2)R^1\\
&\equal{\equref{qt2}}&
\l(hr^2)a^{-1}S(R^2)R^1r^1 =\l(hr^2)a^{-1}ur^1.
\end{eqnarray*}
On the other hand, we compute
\begin{eqnarray*}
\l(R^2h_2)R^1h_1&=&\l(R^2_3h_2)R^1S(R^2_1)R^2_2h_1\\
\mbox{$(\l\in \int_l^{H^*})$}
&=&\l(R^2_2h)R^1S(R^2_1)\\
&\equal{\equref{qt2}}&
\l (R^2h)R^1r^1S(r^2)\\
&\equal{\equref{invr}}&
\l (R^2h)R^1S(S(r^2)r^1)=\l (R^2h)R^1S(u).
\end{eqnarray*}
Now, by \equref{fx1i} and \equref{fx1ii} we have
$\l(hh')=\alpha(h'_2)\l (S^{-2}(h'_1)h)$, for all $h, h'\in H$ or,
equivalently, $\l(h'h)=\alpha ^{-1}(h'_2)\l(hS^2(h'_1))$, for all
$h, h'\in H$, and therefore
\begin{eqnarray*}
(\l\leftharpoonup h)(r^2)a^{-1}ur^1 &=&\l(R^2h)R^1S(u)\\
&=& \alpha^{-1}(R^2_2)\l(hS^2(R^2_1))S(uS^{-1}(R^1))\\
&=&\alpha^{-1}(S^2(R^2_2))\l(hS^2(R^2_1))S(u)S^2(R^1)\\
&\equal{\equref{invr}}&
\alpha^{-1}(R^2_2)\l(hR^2_1)S(u)R^1\\
&\equal{\equref{qt2}}&
\alpha^{-1}(R^2)\l(hr^2)S(u)R^1r^1\\
&=&(\l\leftharpoonup h)(r^2)S(u)b_\alpha  r^1,
\end{eqnarray*}
for all $h\in H$. Since $\hrat$ is dense in $H^*$ in the finite
topology (for example see \cite{dnr}), we may choose $h\in H$ such
that $r^1 (\lambda\leftharpoonup h)(r^2) = r^1\varepsilon(r^2) = 1$,
and then we have that $a^{-1}u = S(u)b_\alpha =v^{-1}b_\alpha$.
Finally, since $v^{-1}$ commutes with grouplikes, we obtain that
$uv=ab_{\alpha}$, so the proof is complete.
\end{proof}

As in the finite dimensional case, \thref{2.2} implies the
following.

\begin{corollary}\colabel{2.5}
Let $(H, R)$ be a co-Frobenius quasitriangular Hopf algebra.
\begin{itemize}
\item[(i)] If $\alpha=\va$ then $S(u)^{-1}u = vu =a$, so $S(u)^{-1}u$ does not depend on $R$.
\item[(ii)] $S(u)=u$ if and only if $a_{\alpha}=a^{-1}$.
\end{itemize}
\end{corollary}

   Every subHopf algebra of a co-Frobenius Hopf
algebra is co-Frobenius \cite[5.3.3]{dnr}. If   $ 0 \neq \lambda \in
\int_l^{H^*}$ and $\lambda(x) \neq 0$ for some $x \in L$, a subHopf
algebra of $H$,  then $\int_l^{H^*} = \int_l^{L^*} = k\lambda$, and
  the modular elements are $a_L = a_H = \lambda(x_1)x_2$, and $\alpha_L =
\alpha_H$ restricted to $L$. However in general the spaces of
integrals do not coincide (in fact, Ker$(\l)$ is a subcoalgebra of
codimension 1) and the modular elements   differ. If $L$ is the
minimal quasitriangular subHopf algebra of $H$ (see
\cite{radfordmin}), then we have the following.

\begin{corollary}  For $(L,R) \subseteq (H,R)$ minimal
quasitriangular, then $a_L = a_H$ if and only if $\alpha_L =
\alpha_H$ restricted to $L$ if and only if $\chi_L = \chi_H$ where
$\chi$ is the generalized Nakayama automorphism defined in
\equref{fx1i}.
\end{corollary}
\begin{proof} By  \thref{2.2}, $a_L = a_H$ if and only if
$a_{\alpha_L} = a_{\alpha_H}$.  Write $R = \sum_{i=1}^m u_i \otimes
v_i$ with the sets of $u_i$ and of $v_i$ linearly independent (cf.
\cite{radfordmin}). The minimal quasitriangular Hopf algebra $L$ is
generated by the $u_i$ and $v_i$. Then $a_{\alpha_L} = a_{\alpha_H}$
if and only if $\alpha_L$ and $\alpha_H$ agree on the $v_i$ and
$b_{\alpha_L} = b_{\alpha_H}$ if and only if $\alpha_L$ and
$\alpha_H$ agree on the $u_i$.  The last equivalence follows from
the definition of the generalized Nakayama automorphism.
\end{proof}

Note that in the above corollary it is important that $(L,R)$ be
minimal quasitriangular.  Let char($k$) be different from 2, and let
$H = H_4$, Sweedler's 4-dimensional Hopf algebra generated by the
grouplike $g$ and the $(1,g)$-primitive $x$. It is well known that
$(H,R)$ is quasitriangular with $R = c+ \xi b$ where $c =
\frac{1}{2}(1 \otimes 1 + 1 \otimes g + g \otimes 1 - g \otimes g)$
and $b =\frac{1}{2}(x \otimes x - x \otimes gx + gx \otimes x + gx
\otimes gx)$. If $\xi \neq 0$, then   $a_H=g$; if $\xi = 0$, $a_H =
1$.  If $\xi \neq 0$, then $\alpha_H(g^ix^j) = \delta_{j,0} (-1)^i$;
if $\xi =0$, then $\alpha_H = \varepsilon$. Now the semisimple
cosemisimple subHopf algebra $L = k[\langle g \rangle]$   has
trivial modular elements and so these agree with those of $H$ if
$\xi = 0$ so that $L$ is minimal quasitriangular, but otherwise they
are different.

\section{$S^4$ when $(H, \sigma)$ is a co-Frobenius coquasitriangular Hopf algebra}\selabel{3}
\setcounter{equation}{0}

Recall that  for $H$ a Hopf algebra and $\sigma: H \otimes
H\rightarrow k$, then $(H, \sigma)$ is called coquasitriangular if
for all $h,l,m \in H$, the map $\sigma$ satisfies:
\begin{eqnarray}
&&\sigma (hl, m)=\sigma(h, m_1)\sigma(l, m_2),\eqlabel{cqt1}\\
&&\sigma(h, lm)=\sigma(h_1, m)\sigma(h_2, l),\eqlabel{cqt2}\\
&&\sigma(h, 1)=\sigma(1, h)=\va(h),\eqlabel{cqt3}\\
&&l_1h_1\sigma(h_2, l_2)=\sigma(h_1, l_1)h_2l_2. \eqlabel{cqt4}
\end{eqnarray}
 These conditions imply that
$\sigma$ is convolution invertible with inverse
$\sigma\circ (S \otimes \Id) = \sigma \circ (\Id \otimes S^{-1})$, and so
$\sigma =\sigma \circ (S\ot S)$.
If only \equref{cqt4} holds and $\sigma$ is invertible, then $(H, \sigma)$ is
called almost commutative.

Let $(H, \sigma)$ be coquasitriangular. Let $u\in H^*$ be defined by
$u(h):=\sigma(h_2, S(h_1))$ with convolution inverse
 $u^{-1}=\sigma(S^2(h_2),h_1)$. Similarly, define $v(h):=(u\circ
S)(h)=\sigma(h_1, S(h_2))$ with convolution inverse
$v^{-1}(h)=\sigma(S^2(h_1), h_2)$. Then by the argument in
\cite{doi}, $S^2 = {\rm co-Inn}_u = {\rm co-Inn}_{v^{-1}} $, that
is,
$$
S^2(h) = u(h_1)h_2 u^{-1}(h_3) = v^{-1}(h_1)h_2 v(h_3).
$$
In particular, $u$ and $v^{-1}$ commute.
\begin{remark}
Note that if $S^2 ={\rm co-Inn}_\omega$ with $\omega(1)=1$, then
$\lambda \circ S^2 = \omega*\lambda*\omega^{-1} =
\omega^{-1}(a^{-1})\l $. But as in \reref{2.1}, $\lambda \circ S^2 =
\alpha(a^{-1})\lambda$ and so $\omega^{-1}(a^{-1}) =
\alpha(a^{-1})$. \qed
\end{remark}

The analogue of \resref{inner} in the coquasitriangular setting
follow.

\begin{remarks}
(i) If $(H,\sigma)$ is almost commutative, then an inner
automorphism induced by a grouplike element is co-inner. For, let $g
\in G(H)$ and let $\Inn_g \in {\rm Aut}(H)$ be the inner
automorphism of $H$ induced by $g$. Then it is easy to check using
\equref{cqt4} that $\Inn_g ={\rm co-Inn}_\omega$ where $\omega \in
\cal{W}_g=\{\sigma( -, g), \sigma^{-1}(g, -), \sigma(g^{-1}, -),
\sigma^{-1}(-, g^{-1}) \}$.

\par (ii) Now let $(H, \sigma)$ be coquasitriangular.
For $g \in G(H)$,  define grouplike elements $\alpha_g$ and
$\beta_g$ in $H^0$ by $\alpha_g = \sigma(-, g^{-1})$ and $\beta_g
=\sigma(g ,-)$. It is easy to check that $\alpha_g$ and $\beta_g$
lie in $G(H^0)$ and that the maps from $G(H)$ to $G(H^0)$, $g\mapsto
\alpha_g$ and $g\mapsto \beta_g$ are well defined group
homomorphisms. Furthermore, here ${\cal W}_g = \{\alpha_g, \beta_g
\}$. \qed
\end{remarks}

Now we have the main result of this section.

\begin{theorem}\thlabel{mainsection3}
Let $(H, \sigma)$ be a co-Frobenius coquasitriangular Hopf algebra
with modular elements  $a \in H$ and $\alpha \in H^*$ and with $u, v$
as defined above. Then
\[
u^{-1}*v =v* u^{-1} = \alpha * \beta_a = \alpha * \alpha_{a}.
\]
\end{theorem}
\begin{proof}
For   $h, l\in H$, we compute $e =\lambda(hl_1)\sigma(l_3,
a^{-1}S(l_2))$ in two different ways.  First, we compute
\[
e \equal{\equref{cqt2}} \lambda(hl_1)\sigma(l_3, S(l_2))\sigma(l_4,
a^{-1}) =((\lambda \leftharpoonup h)*u*\alpha_{a})(l).
\]
Next, we see that
\begin{eqnarray*}
  e & \equal{\equref{a}}& \sigma(l_2, \lambda(h_1l_1)h_2)
\equal{\equref{cqt4}}
\sigma(l_1,h_1)\lambda(l_2h_2)\\
& = &\sigma(l_1 , \lambda(l_4h_2)S(l_2)l_3h_1) =\sigma(l_1,
S(l_2))\lambda(l_3h) =(v*(h \rightharpoonup \lambda))(l).
\end{eqnarray*}
But it is easy to prove directly (or see the second equation in
\equref{b}), that  for all $p\in H^*$, we have  $(\l\lh h)*p=\l\lh
p(S^{-1}(h_2)a^{-1})h_1$, and therefore
\[
(\lambda \leftharpoonup h) * u * \alpha_{a } = \lambda
\leftharpoonup (u*\alpha_a)(S^{-1}(h_2)a^{-1})h_1.
\]
Similarly, since $p*(h\rh \l)=p(S^{-1}(h_1))h_2\rh \l$, for any
$p\in H^*$, we get that
\[
v*(h\rh \l)=(u\circ S)*(h\rh \l)=u(h_1)h_2\rh \l
=\l\lh u(h_1)S^{-2}(h_2)\a(h_3).
\]
Thus $(u * \alpha_a)(S^{-1}(h_2)a^{-1})h_1=
u(h_1)S^{-2}(h_2)\a(h_3)$. The fact that $u*\alpha =
v*\beta_{a^{-1}}$ will follow by applying $\varepsilon$ to this
equality, and using the facts that $v=u\circ S=u\circ S^{-1}$, $a\in
G(H)$ and $\a_a\in G(H^0)$.  For all $h\in H$,
\begin{eqnarray*}
(u * \alpha)(h)&=&(u * \alpha_{a })(S^{-1}(ah))
\\
&=& \alpha_{a^{-1}}(ah_1)v(ah_2)
\hspace{2mm}=\hspace{2mm}
\sigma(ah_1, a)\sigma(ah_2, S(ah_3))\\
&\equal{\equref{cqt2}}&
\sigma(ah_1, S(ah_2)a)
\hspace{2mm}=\hspace{2mm}
\sigma (ah_1, S(h_2))\\
&\equal{\equref{cqt1}}& \sigma(a, S(h_3))\sigma (h_1, S(h_2))
\hspace{2mm}=\hspace{2mm} (v* \beta_{a^{-1}})(h),
\end{eqnarray*}
and thus $u^{-1} * v = \alpha * \beta_a$.
\par Recall that $(H, \tilde{\sigma})$ is also coquasitriangular where
$\tilde{\sigma} = \sigma^{-1} \circ{\rm tw}$, for ${\rm tw}$  the
twist map. It is easily checked that    $\tilde{u} = v^{-1}$,
$\tilde{v} = u^{-1}$, $\widetilde{\alpha_a} = \beta_{a}$, and
$\widetilde{\beta_a} = \alpha_{a}$. Then $\tilde{u}^{-1} * \tilde{v}
= \alpha*\widetilde{\beta_a}$ yields $v * u^{-1} = \alpha
* \alpha_a$, and the proof is complete.
\end{proof}

Dual to \leref{factunim} and \coref{2.5} we have the following.

\begin{corollary}
Let $(H, \sigma)$ be a co-Frobenius coquasitriangular Hopf algebra and $a\in G(H)$
and $\a\in G(H^0)$ as above. Then:
\begin{itemize}
\item[(i)] $\a_a=\b_a$.
\item[(ii)] If $a=1$ then $u^{-1}*v=\a$, so $u^{-1}*(u\circ S)$
does not depend on $\sigma$.
\item[(iii)] $u\circ S=u$ if and only if $\a_a=\a^{-1}$.
\end{itemize}
\end{corollary}
\par Examples of quasitriangular co-Frobenius Hopf algebras include
the cosemisimple Hopf algebra $SL_q(2)$. See \cite[Example
2.12]{bbt} for the computation of the modular elements; the
quasitriangular structure comes from the FRT construction and can be
found, for example, in \cite{kassel}. Another example is the
following.

\begin{example}\label{ore} Let $H$ be the   Hopf algebra generated as an algebra by the
grouplike element $g$ and the $(1,g)$-primitive element $x$ with $gx
= - xg$, $x^2=0$. (See   \cite[Section 4]{bdgn} for details.) Then
$H$ is an infinite dimensional co-Frobenius Hopf algebra with $0
\neq \lambda \in \int^{H^*}_l$ defined by $\lambda(g^ix^j) =
\delta_{i, -1}\delta_{j,1}$. The modular elements are $a=g^{-1}$ and
$\alpha$ defined by $\alpha(g^ix^j) = \delta_{j,0}(-1)^i$. Define
$\sigma: H \otimes H \rightarrow k$ by $\sigma(g^ix^j, g^tx^s) =
\delta_{s,0}\delta_{j,0}(-1)^{it}$. It is straightforward to check
that $(H,\sigma)$ is coquasitriangular. Here $u = v =
u^{-1}=v^{-1}$, $\beta_a(g^ix^j) = \delta_{j,0}(-1)^{-i} =
\delta_{j,0}(-1)^{i} =\alpha_a(g^ix^j)$.
\end{example}

\par Of course, if $(H,R)$ is quasitriangular, and $L$ is a Hopf
algebra contained in $H^*$, then $(L, \sigma)$ is coquasitriangular
where $\sigma(\alpha , \beta) = \alpha(R^1)\beta(R^2)$. An
interesting approach to the converse is via multiplier Hopf
algebras.

\par The results of Drinfeld discussed in this paper have
been extended in \cite{dvdw} to   discrete quasitriangular
multiplier Hopf algebras as defined in \cite{zhang}. A discrete
multiplier Hopf algebra $A$ is one with a cointegral, so if $A$ has
a 1, $A$ is a finite dimensional Hopf algebra.  For $H$  a
co-Frobenius Hopf algebra, then $H^{* {\rm rat}}$ is a regular
discrete multiplier Hopf algebra with integrals, giving interesting
infinite dimensional examples in this context.  The basic theory of
multiplier Hopf algebras is explained in \cite{vd}, that of
  multiplier Hopf algebras with
integrals and their modular elements   in \cite{dvdwradford} where
the analogue of Radford's formula for $S^4$ is proved.

\begin{example} \cite[Section 2]{dvdw} For $H$ as in Example \ref{ore}, let $A =
\{ \lambda \leftharpoonup h | h \in H \} = H^{* \rm{rat}}$. Then it
is shown in \cite{dvdw} that $A$ is a discrete quasitriangular
multiplier Hopf algebra with integrals. Let $\omega_p \in A$ map
$g^ix^j$ to $\delta_{i,p}\delta_{j,0}$. The modular element in
$M(A)$ is $a= \sum_{p \in \mathbb{Z}} (-1)^p \omega_p$; the modular
element in $M(\hat{A})$ is $g^{-1}$. (Recall that by duality for
multiplier Hopf algebras the dual multiplier Hopf algebra $\hat{A}
\cong H$.) The multiplier $R = R^{-1} = \sum_{ p \in \mathbb{Z}} a^p
\otimes \omega_p \in M(A \otimes A)$ makes $(A,R)$ quasitriangular.
  The dual to
$R$ is the coquasitriangular structure from Example \ref{ore}.
\end{example}

For $(H,\sigma)$   coquasitriangular and co-Frobenius, consider the
multiplier Hopf algebra $A= H^{*{\rm rat}}= \lambda \leftharpoonup
H$. Let $R = \sigma \in (H \otimes H)^*$. Since $ \sigma*(\lambda
\leftharpoonup h \otimes \lambda \leftharpoonup l )=
\sigma(h_1,l_1)(\lambda \leftharpoonup h_2 \otimes \lambda
\leftharpoonup l_2)$ and  $
 (\lambda \leftharpoonup h \otimes \lambda \leftharpoonup l )* \sigma
= \sigma(ah_2,al_2)(\lambda \leftharpoonup h_1 \otimes \lambda
\leftharpoonup l_1)$  by \equref{b}, then $\sigma \in M(A \otimes
A):= \{m \in (H \otimes H)^* |m(A \otimes A), (A \otimes A)m
\subseteq A \otimes A \}$. A similar computation shows that
$(\lambda \leftharpoonup h \otimes \varepsilon)*\sigma \in A \otimes
M(A)$ with similar variations on the left and right.   The defining
properties for quasitriangularity \cite[2.3]{dvdw} then appear to
follow from the properties of $\sigma$.

\end{document}